\documentclass[reqno]{amsart}
\usepackage{amssymb,mathtools,mathrsfs}
\usepackage{ifpdf}
\ifpdf
 \usepackage[hyperindex]{hyperref}
\else
 \expandafter\ifx\csname dvipdfm\endcsname\relax
 \usepackage[hypertex,hyperindex]{hyperref}
 \else
 \usepackage[dvipdfm,hyperindex]{hyperref}
 \fi
\fi
\allowdisplaybreaks[4]
\numberwithin{equation}{section}
\theoremstyle{plain}
\newtheorem{theorem}{Theorem}[section]
\newtheorem{lemma}{Lemma}[section]

\theoremstyle{remark}
\newtheorem{remark}{Remark}[section]
\DeclareMathOperator{\td}{d\mspace{-2mu}}

\begin{document}

\title[Inequalities to logarithmically complete monotonicity]
{From inequalities involving exponential functions and sums to logarithmically complete monotonicity of ratios of gamma functions}

\author[F. Qi]{Feng Qi}
\address{College of Mathematics, Inner Mongolia University for Nationalities, Tongliao 028043, Inner Mongolia, China; School of Mathematical Sciences, Tianjin Polytechnic University, Tianjin 300387, China}
\email{\href{mailto: F. Qi <qifeng618@gmail.com>}{qifeng618@gmail.com}, \href{mailto: F. Qi <qifeng618@hotmail.com>}{qifeng618@hotmail.com}, \href{mailto: F. Qi <qifeng618@qq.com>}{qifeng618@qq.com}}
\urladdr{\url{https://qifeng618.wordpress.com}}

\author[B.-N. Guo]{Bai-Ni Guo}
\address[Guo]{School of Mathematics and Informatics, Henan Polytechnic University, Jiaozuo 454010, Henan, China}
\email{\href{mailto: B.-N. Guo <bai.ni.guo@gmail.com>}{bai.ni.guo@gmail.com}, \href{mailto: B.-N. Guo <bai.ni.guo@hotmail.com>}{bai.ni.guo@hotmail.com}}
\urladdr{\url{http://www.researcherid.com/rid/C-8032-2013}}

\begin{abstract}
In the paper, the authors review origins, motivations, and generalizations of a series of inequalities involving several exponential functions and sums, establish three new inequalities involving finite exponential functions and sums by finding convexity of a function related to the generating function of the Bernoulli numbers, survey the history, backgrounds, generalizations, logarithmically complete monotonicity, and applications of a series of ratios of finite gamma functions, present complete monotonicity of a linear combination of finite trigamma functions, construct a new ratio of finite gamma functions, derives monotonicity, logarithmic convexity, concavity, complete monotonicity, and the Bernstein function property of the newly constructed ratio of finite gamma functions, and suggest two linear combinations of finite trigamma functions and two ratios of finite gamma functions to be investigated.
\end{abstract}



\thanks{This paper was typeset using \AmS-\LaTeX}

\maketitle
\tableofcontents

\section{Motivations}

In~\cite[Lemma~1]{Alzer-JMAA-2018}, by convexity of $\frac{1}{y^{1/\lambda}-1}$ and $\frac{1}{y^{1/\lambda}-1}+\frac{1}{y^{1/(1-\lambda)}-1}$ with respect to $\lambda\in(0,1)$, the inequality
\begin{equation}\label{Alzer-jmaa-2018Lem1}
\frac{1}{y-1}-\frac{1}{y^{1/\lambda}-1}-\frac{1}{y^{1/(1-\lambda)}-1}>0, \quad y>1
\end{equation}
was proved to be true.
\par
In~\cite[Lemma~1.4]{Ouimet-JMAA-2018}, inequality~\eqref{Alzer-jmaa-2018Lem1} was generalized as one which can be reformulated as
\begin{equation}\label{Ouimet-Ineq-exp-2018}
\frac1{y-1}-\sum_{k=1}^n\frac{1}{y^{1/\lambda_k}-1}>0,
\end{equation}
where $y>1$ and $\lambda_1,\lambda_2,\dotsc,\lambda_n\in(0,1)$ such that $\sum_{k=1}^{n}\lambda_k=1$.
\par
In proofs of~\cite[Theorems~2.1 and~2.2]{Alzer-CM-JMAA.tex},  by convexity of the function $\frac{1}{y^{1/\lambda}-1}$ for $y>1$ with respect to $\lambda\in(0,\infty)$, by convexity of the function $t H\bigl(\frac1t\bigr)$ on $(0,\infty)$, and by decreasing monotonicity of $H(t)$ on $(-\infty,\infty)$, where
$$
H(t)=\frac{t}{e^t-1}, \quad t\in\mathbb{R}
$$
is the generating function of the Bernoulli numbers (see~\cite[Section~1]{CAM-D-18-00067.tex} and~\cite[Chapter~1]{Temme-96-book}), the inequality
\begin{equation}\label{Qi-Exp-Ineq-sum-O}
\Biggl(\sum_{k=1}^n\lambda_k\Biggr) H\biggl(\frac{x}{\sum_{k=1}^n\lambda_k}\biggr) \ge\sum_{k=1}^{n}\lambda_k H\biggl(\frac{x}{\lambda_k}\biggr)
\end{equation}
for $\lambda_k>0$ and $x>0$ was proved to be true.
\par
In the proof of~\cite[Theorem~3.1]{Q-Alzer-CM-Q.tex}, by convexity of the function $t^3H\bigl(\frac1t\bigr)$ on $(0,\infty)$, the inequality
\begin{equation}\label{Qi-Exp-Ineq-sum-3}
\Biggl(\sum_{k=1}^n\lambda_k\Biggr)^3 H\biggl(\frac{x}{\sum_{k=1}^n\lambda_k}\biggr)
\ge\sum_{k=1}^{n}\lambda_k^3 H\biggl(\frac{x}{\lambda_k}\biggr)
\end{equation}
was proved to be true for $\lambda_k>0$ and $x>0$.
\par
In~\cite[Lemma~A.1]{OUIMET-ARXIV-1907-05262}, the inequality
\begin{equation}\label{Ouimet-lem-ineq}
\sum_{i=1}^m\frac1{y^{1/\nu_i}-1}+\sum_{j=1}^n\frac1{y^{1/\tau_j}-1}
>\sum_{i=1}^m\sum_{j=1}^n\frac1{y^{1/\lambda_{ij}}-1}
\end{equation}
was complicatedly proved to be valid for $y>1$ and $0<\lambda_{ij}\le1$, where $\nu_i=\sum_{j=1}^n\lambda_{ij}$ and $\tau_j=\sum_{i=1}^m\lambda_{ij}$ satisfying $\sum_{i=1}^m\nu_i=\sum_{j=1}^n\tau_j=1$.
\par
In~\cite[Theorem~3.1]{Ouimet-LCM-BKMS.tex}, by considering convexity of the function $tH\bigl(\frac1t\bigr)$ on $(0,\infty)$, the inequality
\begin{equation}\label{Ouimet-lem-iq}
\sum_{i=1}^m\frac1{e^{x/\nu_i}-1}+\sum_{j=1}^n\frac1{e^{x/\tau_j}-1}
\ge2\sum_{i=1}^m\sum_{j=1}^n\frac1{e^{x/\lambda_{ij}}-1}
\end{equation}
was proved to be true for $x>0$ and $\lambda_{ij}>0$, where $\nu_i=\sum_{j=1}^n\lambda_{ij}$ and $\tau_j=\sum_{i=1}^m\lambda_{ij}$.
\par
We observe that
\begin{enumerate}
\item
inequality~\eqref{Alzer-jmaa-2018Lem1} can be rearranged as
\begin{equation}\label{Alzer-jmaa-2018Lem1-rearr}
\frac{1}{e^{x/[\lambda+(1-\lambda)]}-1}>\frac{1}{e^{x/\lambda}-1}+\frac{1}{e^{x/(1-\lambda)}-1},
\end{equation}
where $x=\ln y>0$ and $\lambda\in(0,1)$;
\item
inequality~\eqref{Ouimet-Ineq-exp-2018} can be rewritten as
\begin{equation}\label{Ouimet-Ineq-exp-2018-rew}
\frac1{e^{x/\sum_{k=1}^{n}\lambda_k}-1}>\sum_{k=1}^n\frac{1}{e^{x/\lambda_k}-1},
\end{equation}
where $x=\ln y>0$ and $\lambda_1,\lambda_2,\dotsc,\lambda_n\in(0,1)$ such that $\sum_{k=1}^{n}\lambda_k=1$;
\item
inequality~\eqref{Qi-Exp-Ineq-sum-O} can be reformulated as~\eqref{Ouimet-Ineq-exp-2018-rew} without restrictions $\sum_{k=1}^{n}\lambda_k=1$ and $\lambda_1,\lambda_2,\dotsc,\lambda_n<1$;
\item
inequality~\eqref{Qi-Exp-Ineq-sum-3} can be rewritten as
\begin{equation}\label{Qi-Exp-Ineq-sum-3-rew}
\frac{\bigl(\sum_{k=1}^n\lambda_k\bigr)^2}{e^{x/\sum_{k=1}^n\lambda_k}-1}
\ge\sum_{k=1}^{n} \frac{\lambda_k^2}{e^{x/\lambda_k}-1},
\end{equation}
where $\lambda_k>0$ and $x>0$.
\item
inequality~\eqref{Ouimet-lem-ineq} can be reformulated as
\begin{equation}\label{Ouimet-lem-iq-rew}
\sum_{i=1}^m\frac1{e^{x/\nu_i}-1}+\sum_{j=1}^n\frac1{e^{x/\tau_j}-1}
\ge\sum_{i=1}^m\sum_{j=1}^n\frac1{e^{x/\lambda_{ij}}-1}
\end{equation}
for $x=\ln y>0$ and $0<\lambda_{ij}\le1$, where $\nu_i=\sum_{j=1}^n\lambda_{ij}$ and $\tau_j=\sum_{i=1}^m\lambda_{ij}$ satisfying $\sum_{i=1}^m\nu_i=\sum_{j=1}^n\tau_j=1$.
\item
both proofs in~\cite[Lemma~1.4]{Ouimet-JMAA-2018} and~\cite[Lemma~A.1]{OUIMET-ARXIV-1907-05262} for inequalities~\eqref{Ouimet-Ineq-exp-2018} and~\eqref{Ouimet-lem-ineq} are not convincible;
\item
inequality~\eqref{Ouimet-lem-iq} refines~\eqref{Ouimet-lem-ineq} and~\eqref{Ouimet-lem-iq-rew} and removes off the restrictions $\lambda_{ij}\le1$ and $\sum_{i=1}^m\nu_i=\sum_{j=1}^n\tau_j=1$ appeared in~\cite[Lemma~A.1]{OUIMET-ARXIV-1907-05262};
\item
inequality~\eqref{Ouimet-Ineq-exp-2018-rew} generalizes~\eqref{Alzer-jmaa-2018Lem1-rearr};
\item
when taking $m=n$ and $\lambda_{1i}=\lambda_{i1}>0$ for $1\le i\le n$ and letting $\lambda_{ij}\to0^+$ for $2\le i,j\le n$, inequality~\eqref{Ouimet-lem-iq} becomes
\begin{gather*}
\frac1{e^{x/\sum_{j=1}^n\lambda_{1j}}-1}+\sum_{i=2}^n\frac1{e^{x/\lambda_{i1}}-1} +\frac1{e^{x/\sum_{i=1}^n\lambda_{i1}}-1} +\sum_{j=2}^n\frac1{e^{x/\lambda_{1j}}-1}\\
\ge2\sum_{i=1}^n\frac1{e^{x/\lambda_{i1}}-1}+2\sum_{i=1}^n\frac1{e^{x/\lambda_{i2}}-1} +\dotsm+2\sum_{i=1}^n\frac1{e^{x/\lambda_{in}}-1}\\
=2\sum_{i=1}^n\frac1{e^{x/\lambda_{i1}}-1}+\frac2{e^{x/\lambda_{12}}-1}+\frac2{e^{x/\lambda_{13}}-1}+\dotsm+\frac2{e^{x/\lambda_{1n}}-1}
\end{gather*}
which can be simplified as
\begin{equation}\label{mu-ij>2=0-Eq}
\frac1{e^{x/\sum_{i=1}^n\lambda_{i1}}-1}
\ge\sum_{i=1}^n\frac1{e^{x/\lambda_{i1}}-1};
\end{equation}
this inequality is equivalent to~\eqref{Ouimet-Ineq-exp-2018-rew} without restrictions $\sum_{k=1}^{n}\lambda_k=1$ and  $\lambda_1,\lambda_2,\dotsc,\lambda_n<1$.
\end{enumerate}
In a word, inequality~\eqref{Ouimet-lem-iq} established in~\cite[Theorem~3.1]{Ouimet-LCM-BKMS.tex} extends, generalizes, and refines all of the above inequalities other than~\eqref{Qi-Exp-Ineq-sum-3} and~\eqref{Qi-Exp-Ineq-sum-3-rew}.
\par
Motivated by inequalities~\eqref{Qi-Exp-Ineq-sum-3} and~\eqref{Qi-Exp-Ineq-sum-3-rew}, we would like to ask a question: what is the largest range of $\alpha$ such that
\begin{equation}\label{Qi-Exp-Ineq-quest}
\frac{\bigl(\sum_{k=1}^n\lambda_k\bigr)^\alpha}{e^{x/\sum_{k=1}^n\lambda_k}-1}
\ge\sum_{k=1}^{n} \frac{\lambda_k^\alpha}{e^{x/\lambda_k}-1}
\end{equation}
validates for $x>0$ and $\lambda_k>0$?
\par
Motivated by inequalities~\eqref{Ouimet-lem-iq} and~\eqref{Qi-Exp-Ineq-quest}, we would like to ask a question: what are the largest ranges of $\alpha$ and $\rho$ such that
\begin{equation}\label{epsilon-varpsilon-ineq}
\sum_{i=1}^m\frac{\nu_i^\alpha}{e^{x/\nu_i}-1}+\sum_{j=1}^n\frac{\tau_j^\alpha}{e^{x/\tau_j}-1}
\ge\rho\sum_{i=1}^m\sum_{j=1}^n\frac{\lambda_{ij}^\alpha}{e^{x/\lambda_{ij}}-1}
\end{equation}
is valid for $x>0$ and $\lambda_{ij}>0$? where $\nu_i=\sum_{j=1}^n\lambda_{ij}$ and $\tau_j=\sum_{i=1}^m\lambda_{ij}$.
\par
Motivated by proofs of inequalities~\eqref{Qi-Exp-Ineq-sum-O}, \eqref{Qi-Exp-Ineq-sum-3}, and~\eqref{Ouimet-lem-iq} in the papers~\cite{Q-Alzer-CM-Q.tex, Ouimet-LCM-BKMS.tex, Alzer-CM-JMAA.tex}, we would like to ask a question: what is the largest range of $\alpha$ such that the function $t^\alpha H\bigl(\frac1t\bigr)$ is convex on $(0,\infty)$?

\section{Lemmas}

The following lemmas are useful in this paper.

\begin{lemma}[\cite{avv1}]\label{le:2.1}
For $a,b\in\mathbb{R}$ with $a<b$, let $U(t)$ and $V(t)$ be continuous on $[a,b]$, differentiable on $(a,b)$, and $V'(t)\ne 0$ on $(a,b)$. If $\frac{U'(t)}{V'(t)}$ is decreasing on $(a,b)$, then the functions
$$
\mathcal{F}(t)=\frac{U(t)-U(a)}{V(t)-V(a)}\quad\text{and}\quad \mathcal{G}(t)=\frac{U(t)-U(b)}{V(t)-V(b)}
$$
are both decreasing on $(a,b)$.
\end{lemma}

A function $\varphi:[0,\infty)\to\mathbb{R}$ is said to be star-shaped if $\varphi(\nu t)\le\nu\varphi(t)$ for $\nu\in[0,1]$ and $t\ge0$.
A real function $\varphi$ defined on a set $S\subset\mathbb{R}^n$ is said to be super-additive if $s,t\in S$ implies $s+t\in S$ and $\varphi(s+t)\ge\varphi(s)+\varphi(t)$. See~\cite[Chapter~16]{Marshall-Olkin-Arnold} and~\cite[Section~3.4]{Niculescu-Persson-Monograph-2018}.

\begin{lemma}[{\cite[pp.~650--651, Section~B.9]{Marshall-Olkin-Arnold}}]\label{MOA-B.9-650-651}
Among convex functions, star-shaped functions, and super-additive functions, these are the following relations:
\begin{enumerate}
\item
if $\varphi$ is convex on $[0,\infty)$ with $\varphi(0)\le0$, then $\varphi$ is star-shaped;
\item
if $\varphi:[0,\infty)\to\mathbb{R}$ is star-shaped, then $\varphi$ is super-additive.
\end{enumerate}
\end{lemma}

\section{Convexity and logarithmic concavity of a function related to generating function of Bernoulli numbers}

Now we give an answer to the third question above and find something more.

\begin{theorem}\label{t-alpha-BGF-convex-thm}
Let $\alpha\in\mathbb{R}$ and
\begin{equation*}
\mathfrak{H}_\alpha(t)=t^\alpha H\biggl(\frac1t\biggr)=\frac{t^{\alpha-1}}{e^{1/t}-1}, \quad t\in(0,\infty).
\end{equation*}
Then
\begin{enumerate}
\item
if $\alpha\ge1$, the function $\mathfrak{H}_\alpha(t)$ is convex on $(0,\infty)$;
\item
if $0\le\alpha<1$, the function $\mathfrak{H}_\alpha(t)$ has a unique inflection point on $(0,\infty)$;
\item
if $\alpha<0$, the function $\mathfrak{H}_\alpha(t)$ has only two inflection points on $(0,\infty)$;
\item
the function $\mathfrak{H}_\alpha(t)$ has the limits
\begin{equation}\label{frak-H-limit=020}
\lim_{t\to0^+}\mathfrak{H}_\alpha(t)=0, \quad \alpha\in\mathbb{R}
\end{equation}
and
\begin{equation}\label{lim-infty-eq}
\lim_{t\to\infty}\mathfrak{H}_\alpha(t)
=\begin{dcases}
\infty, & \alpha>0;\\
1, & \alpha=0;\\
0, & \alpha<0.
\end{dcases}
\end{equation}
\end{enumerate}
\end{theorem}

\begin{proof}
By direct computation, we have
\begin{equation*}
\frac{\td\mathfrak{H}_\alpha(t)}{\td t}=\frac{t^{\alpha-3} \bigl([(\alpha-1)t+1]e^{1/t}+(1-\alpha)t\bigr)}{(e^{1/t}-1)^2}
\end{equation*}
and
\begin{align*}
\frac{\td{}^2\mathfrak{H}_\alpha(t)}{\td t^2}
&=\frac{t^{\alpha-5}
\left(\begin{gathered}(\alpha-1)(\alpha-2)t^2\\
-\bigl[2(\alpha-1)(\alpha-2)t^2+2(\alpha-2)t-1\bigr]e^{1/t}\\
+\bigl[(\alpha-1)(\alpha-2)t^2+2(\alpha-2)t+1\bigr]e^{2/t}\end{gathered}\right)} {(e^{1/t}-1)^3}\\
&\triangleq\frac{t^{\alpha-5}}{(e^{1/t}-1)^3}H_\alpha\biggl(\frac{1}{t}\biggr),
\end{align*}
where
\begin{align*}
H_\alpha(t)&=\frac{\Biggl(\begin{gathered}
(\alpha-1)(\alpha-2)+\bigl[t^2-2(\alpha-2)t-2(\alpha-2)(\alpha-1)\bigr]e^t\\
+\bigl[t^2+2(\alpha-2)t+(\alpha-1)(\alpha-2)\bigr]e^{2t}\end{gathered}\Biggr)}{t^2}\\
&=\frac{1}{t^2}\Biggl[\alpha-\frac32+\frac{2te^t+\sqrt{(4t+1)e^{2t}-2(2t^2+2t+1)e^t+1}\,}{2(e^t-1)}\Biggr]\\
&\quad\times\Biggl[\alpha-\frac32+\frac{2te^t-\sqrt{(4t+1)e^{2t}-2(2t^2+2t+1)e^t+1}\,}{2(e^t-1)}\Biggr]\\
&\triangleq\frac{1}{t^2}\biggl[\alpha-\frac32+\frac{\mathcal{H}_1(t)}2\biggr] \biggl[\alpha-\frac32+\frac{\mathcal{H}_2(t)}2\biggr]
\end{align*}
and
\begin{gather*}
(4t+1)e^{2t}-2(2t^2+2t+1)e^t+1=t^2+\sum_{k=3}^{\infty}\big[2k\bigl(2^k-2k\bigr)+2^k-2\bigr]\frac{t^k}{k!}\\
=t^2+3t^3+\frac{13t^4}{4}+\frac{25t^5}{12}+\frac{343t^6}{360}+\frac{41t^7}{120} +\frac{2047t^8}{20160}+\frac{1567t^9}{60480}+\dotsm>0.
\end{gather*}
By straightforward calculation, we have
\begin{equation*}
\lim_{t\to0}\mathcal{H}_1(t)=3, \quad \lim_{t\to\infty}\mathcal{H}_1(t)=\infty, \quad \lim_{t\to0}\mathcal{H}_2(t)=1, \quad \lim_{t\to\infty}\mathcal{H}_2(t)=\infty.
\end{equation*}
\par
Since
\begin{gather*}
\frac{\bigl[2te^t+\sqrt{(4t+1)e^{2t}-2(2t^2+2t+1)e^t+1}\,\bigr]'}{(e^t-1)'}\\
=\frac{(4t+3)e^t-(2t^2+6t+3)}{\sqrt{(4t+1)e^{2t}-2(2t^2+2t+1)e^t+1}\,}+2t+2,\\
(4t+3)e^t-(2t^2+6t+3)=(2t+1)t+(4t+3)\bigl(e^t-1-t\bigr)>0,\\
\frac{\bigl([(4t+3)e^t-(2t^2+6t+3)]^2\bigr)'}{[(4t+1)e^{2t}-2(2t^2+2t+1)e^t+1]'}
=4t+7-2\frac{2t+3}{e^t},
\end{gather*}
and
$$
\frac{\td}{\td t}\biggl(\frac{2t+3}{e^t}\biggr)=-\frac{2t+1}{e^t},
$$
making use of Lemma~\ref{le:2.1} twice, we can deduce that the function $\mathcal{H}_1(t)$ is increasing on $(0,\infty)$.
\par
Since
\begin{gather*}
\Biggl[\frac{2te^t-\sqrt{(4t+1)e^{2t}-2(2t^2+2t+1)e^t+1}\,}{e^t-1}\Biggr]'\\
=\frac{2e^t\bigl(e^t-t-1\bigr)\left[\begin{gathered}\sqrt{(4t+1)e^{2t}-2(2t^2+2t+1)e^t+1}\,\\
-\frac{e^{2t}+\bigl(t^2-3t-2\bigr)e^t+t^2+3t+1}{e^t-t-1}\end{gathered}\right]} {(e^t-1)^2\sqrt{(4t+1)e^{2t}-2(2t^2+2t+1)e^t+1}\,},\\
e^{2t}+\bigl(t^2-3t-2\bigr)e^t+t^2+3t+1
=\sum_{k=3}^{\infty}\bigl[2\bigl(2^{k-1}-1\bigr)+k(k-4)\bigr]\frac{t^k}{k!}\\
=\frac{t^3}{2}+\frac{7t^4}{12}+\frac{7t^5}{24}+\frac{37t^6}{360}+\frac{7t^7}{240}+\frac{143t^8}{20160}+\frac{37t^9}{24192}+\dotsm>0,
\end{gather*}
and
\begin{gather*}
\Bigl[\sqrt{(4t+1)e^{2t}-2(2t^2+2t+1)e^t+1}\,\Bigr]^2
-\frac{\left[\begin{gathered}e^{2t}+\bigl(t^2-3t-2\bigr)e^t\\ +t^2+3t+1\end{gathered}\right]^2}{(e^t-t-1)^2}\\
=\frac{t\Biggl[\begin{gathered}4e^{4t}-2(7t+4)e^{3t}-t\bigl(t^2-18t-18\bigr)e^{2t}\\
-\bigl(6t^3+12t^2-6t-8\bigr)e^t-t^3-6t^2-10t-4\end{gathered}\Biggr]}{(e^t-t-1)^2}\\
=\frac{t}{24(e^t-t-1)^2}\sum_{k=6}^{\infty}\left[\begin{gathered}192+144k+144k^2-144k^3\\
-k\bigl(3k^2-117k-102\bigr)2^k\\
-(112k+192)3^k+96\times4^k\end{gathered}\right]\frac{t^k}{k!}\\
=\frac{t}{24(e^t-t-1)^2}\sum_{k=6}^{\infty}
\left[\begin{gathered}16\bigl(2^{2k}-9k^3\bigr)+144k^2+144k+192\\
+\bigl(12\times2^k-3k^3\bigr)2^k+k(117k+102)2^k\\
+68\times4^k-16(7k+12)3^k\end{gathered}\right]\frac{t^k}{k!}\\
=\frac{t}{(e^t-t-1)^2}\biggl(\frac{t^6}{3}+\frac{3t^7}{4}+\frac{17t^8}{20}+\frac{1403t^9}{2160}+\frac{1433t^{10}}{3780}+\dotsm\biggr)
>0
\end{gather*}
on $(0,\infty)$, the function $\mathcal{H}_2(t)$ is increasing on $(0,\infty)$.
\par
From the above increasing monotonicity of $\mathcal{H}_1(t)$ and $\mathcal{H}_2(t)$ on $(0,\infty)$, it follows that, if and only if $\alpha\ge1$, the function $H_\alpha(t)$ is positive on $(0,\infty)$. Therefore, if and only if $\alpha\ge1$, the second derivative $\frac{\td{}^2\mathfrak{H}_\alpha(t)}{\td t^2}$ is positive on $(0,\infty)$. Consequently, if and only if $\alpha\ge1$, the function $\mathfrak{H}_\alpha(t)$ is convex on $(0,\infty)$.
\par
The proof of the existence of inflection points of the function $\mathfrak{H}_\alpha(t)$ on $(0,\infty)$ is straightforward.
\par
It is easy to see
\begin{equation*}
\lim_{t\to0^+}\mathfrak{H}_\alpha(t)=\lim_{t\to0^+}\frac{t^{\alpha-1}}{e^{1/t}-1}
=\lim_{s\to\infty}\frac{s^{1-\alpha}}{e^s-1}=0.
\end{equation*}
Since
\begin{equation}\label{2lim-0-infty}
t\bigl(e^{1/t}-1\bigr)
\to
\begin{dcases}
1, & t\to\infty\\
\infty, & t\to0^+
\end{dcases}
\end{equation}
and
\begin{equation*}
\mathfrak{H}_\alpha(t)=\frac{t^{\alpha}}{[t(e^{1/t}-1)]}, \quad t\in(0,\infty),
\end{equation*}
the limits in~\eqref{lim-infty-eq} follow immediately. The proof of Theorem~\ref{t-alpha-BGF-convex-thm} is complete.
\end{proof}

\begin{theorem}\label{log-concave-thm}
Let $\alpha\in\mathbb{R}$. Then
\begin{enumerate}
\item
if $\alpha\ge0$, the function $\mathfrak{H}_\alpha(t)$ is logarithmically concave on $(0,\infty)$;
\item
if $\alpha<0$, the logarithm $\ln\mathfrak{H}_\alpha(t)$ has a unique inflection point on $(0,\infty)$;
\item
if $\alpha\ge0$, the function $\mathfrak{H}_\alpha(t)$ is increasing on $(0,\infty)$;
\item
if $\alpha<0$, the function $\mathfrak{H}_\alpha(t)$ has a unique maximum on $(0,\infty)$.
\end{enumerate}
\end{theorem}

\begin{proof}
By standard computation and by virtue of~\eqref{2lim-0-infty}, we have
\begin{equation}\label{first-deriv-lim}
\frac{\td{}\ln\mathfrak{H}_\alpha(t)}{\td t}
=
\begin{dcases}
\frac{1}{(1-e^{-1/t})t^2}+\frac{a-1}{t} \to\infty, & t\to0^+\\
\frac{e^{1/t}}{[t(e^{1/t}-1)]t}+\frac{a-1}{t} \to 0, & t\to\infty
\end{dcases}
\end{equation}
and
\begin{align*}
\frac{\td{}^2\ln\mathfrak{H}_\alpha(t)}{\td t^2}
&=\frac{1}{t^2}\biggl[\frac{t^2+\bigl(1+2t-2t^2\bigr)e^{1/t}+t(t-2)e^{2/t}}{(e^{1/t}-1)^2 t^2}-\alpha\biggr]\\
&\triangleq\frac{1}{t^2}\biggl[\mathscr{H}\biggl(\frac1t\biggr)-\alpha\biggr],
\end{align*}
where
\begin{equation*}
\mathscr{H}(t)=\frac{1+\bigl(t^2+2t-2\bigr)e^t+(1-2t)e^{2t}}{(e^t-1)^2}
\to
\begin{dcases}
0 & t\to 0^+;\\
-\infty, & t\to\infty.
\end{dcases}
\end{equation*}
By standard calculation, we have
\begin{gather*}
\frac{[1+(t^2+2t-2)e^t+(1-2t)e^{2t}]'}{[(e^t-1)^2]'}=\frac{t(4+t-4e^t)}{2(e^t-1)},\\
\frac{[t(4+t-4e^t)]'}{(e^t-1)'}=2\biggl[\frac{2+t}{e^t}-2(t+1)\biggr], \quad \biggl(\frac{2+t}{e^t}\biggr)'=-\frac{1+t}{e^t}<0.
\end{gather*}
Employing Lemma~\ref{le:2.1} twice arrives at that the function $\mathscr{H}(t)$ is decreasing on $(0,\infty)$. Then the function $\mathscr{H}\bigl(\frac1t\bigr)$ is increasing on $(0,\infty)$, with the limits
\begin{equation}\label{second-H-deriv-lim}
\lim_{t\to0^+}\mathscr{H}\biggl(\frac1t\biggr)=-\infty \quad\text{and}\quad \lim_{t\to\infty}\mathscr{H}\biggl(\frac1t\biggr)=0.
\end{equation}
Accordingly, the function $\mathscr{H}\bigl(\frac1t\bigr)$ is negative on $(0,\infty)$ and, if $\alpha\ge0$, the second derivative $\frac{\td{}^2\ln\mathfrak{H}_\alpha(t)}{\td t^2}$ is negative on $(0,\infty)$. Hence, if $\alpha\ge0$, the function $\mathfrak{H}_\alpha(t)$ is logarithmically concave on $(0,\infty)$ and the first derivative $\frac{\td{}\ln\mathfrak{H}_\alpha(t)}{\td t}$ is decreasing on $(0,\infty)$. Combining this with the limits in~\eqref{first-deriv-lim} reveals that, if $\alpha\ge0$, the first derivative $\frac{\td{}\ln\mathfrak{H}_\alpha(t)}{\td t}$ is positive on $(0,\infty)$. This means that, if $\alpha\ge0$, the function $\mathfrak{H}_\alpha(t)$ is increasing on $(0,\infty)$.
\par
It is not difficult to see that, if $\alpha<0$, by the limits in~\eqref{first-deriv-lim} and~\eqref{second-H-deriv-lim}, the second derivative $\frac{\td{}^2\ln\mathfrak{H}_\alpha(t)}{\td t^2}$ has a zero, the first derivative $\frac{\td{}\ln\mathfrak{H}_\alpha(t)}{\td t}$ has only one zero, and the function $\ln\mathfrak{H}_\alpha(t)$ has only one inflection point and has only one maximum point on $(0,\infty)$.
The proof of Theorem~\ref{log-concave-thm} is complete.
\end{proof}

\begin{remark}
It is well known~\cite[Section~1.3]{Niculescu-Persson-Monograph-2018} that a logarithmically convex function must be convex, but not conversely. It is also well known~\cite[Section~1.3]{Niculescu-Persson-Monograph-2018} that a concave function must be logarithmically concave, but not conversely. The function $\mathfrak{H}_\alpha(t)$ is an example that a logarithmically concave function may not be concave, that a convex function may not be logarithmically convex, and so on.
\end{remark}

\section{Three new inequalities involving exponential functions and sums}

Making use of some conclusions in Theorems~\ref{t-alpha-BGF-convex-thm} and~\ref{log-concave-thm}, we now start out to derive several inequalities involving exponential functions and sums and to answer the first and second questions above.

\begin{theorem}\label{alpha>=1-thm}
For $\alpha\ge1$, $x>0$, and $\lambda_{ij}>0$ for $1\le i\le m$ and $1\le j\le n$, denote $\nu_i=\sum_{j=1}^n\lambda_{ij}$ and $\tau_j=\sum_{i=1}^m\lambda_{ij}$. Then
\begin{equation}\label{alpha>=1-inequal}
\sum_{i=1}^m\frac{\nu_i^{\alpha-1}}{e^{x/\nu_i}-1}+ \sum_{j=1}^n\frac{\tau_j^{\alpha-1}}{e^{x/\tau_j}-1}
\ge2\sum_{i=1}^m\sum_{j=1}^n\frac{\lambda_{ij}^{\alpha-1}}{e^{x/\lambda_{ij}}-1}.
\end{equation}
\end{theorem}

\begin{proof}
Combining the first conclusion and the limit~\eqref{frak-H-limit=020} in Theorem~\ref{t-alpha-BGF-convex-thm} with Lemma~\ref{MOA-B.9-650-651} yields that, if $\alpha\ge1$, the function $\mathfrak{H}_\alpha(t)$ with redefining $\mathfrak{H}_\alpha(0)=0$ is convex, then star-shaped, and then supper-additive on $[0,\infty)$. Consequently, it follows that
\begin{equation*}
\mathfrak{H}_\alpha\biggl(\frac{\nu_i}{x}\biggr)
=\mathfrak{H}_\alpha\Biggl(\frac{\sum_{j=1}^n\lambda_{ij}}{x}\Biggr)
\ge\sum_{j=1}^n\mathfrak{H}_\alpha\biggl(\frac{\lambda_{ij}}{x}\biggr)
\end{equation*}
and
\begin{equation*}
\mathfrak{H}_\alpha\biggl(\frac{\tau_j}{x}\biggr)
=\mathfrak{H}_\alpha\Biggl(\frac{\sum_{i=1}^m\lambda_{ij}}{x}\Biggr)
\ge\sum_{i=1}^m\mathfrak{H}_\alpha\biggl(\frac{\lambda_{ij}}{x}\biggr).
\end{equation*}
Accordingly, we obtain
\begin{align*}
\sum_{i=1}^m\mathfrak{H}_\alpha\biggl(\frac{\nu_i}{x}\biggr)+ \sum_{j=1}^n\mathfrak{H}_\alpha\biggl(\frac{\tau_j}{x}\biggr)
&\ge\sum_{i=1}^m\sum_{j=1}^n\mathfrak{H}_\alpha\biggl(\frac{\lambda_{ij}}{x}\biggr) +\sum_{j=1}^n\sum_{i=1}^m\mathfrak{H}_\alpha\biggl(\frac{\lambda_{ij}}{x}\biggr)\\
&=2\sum_{i=1}^m\sum_{j=1}^n\mathfrak{H}_\alpha\biggl(\frac{\lambda_{ij}}{x}\biggr)
\end{align*}
which can be rearranged as
\begin{equation*}
\sum_{i=1}^m\frac{(\nu_i/x)^{\alpha-1}}{e^{x/\nu_i}-1}+ \sum_{j=1}^n\frac{(\tau_j/x)^{\alpha-1}}{e^{x/\tau_j}-1}
\ge2\sum_{i=1}^m\sum_{j=1}^n\frac{(\lambda_{ij}/x)^{\alpha-1}}{e^{x/\lambda_{ij}}-1}.
\end{equation*}
The proof of Theorem~\ref{alpha>=1-thm} is complete.
\end{proof}

\begin{remark}
As the deduction of~\eqref{mu-ij>2=0-Eq}, setting $m=n$ and $\lambda_{1k}=\lambda_{k1}=\lambda_k>0$ for $1\le k\le n$ and letting $\lambda_{ij}\to0^+$ for $2\le i,j\le n$ in inequality~\eqref{alpha>=1-inequal} result in inequality~\eqref{Qi-Exp-Ineq-quest} for $\alpha\ge0$.
\par
The inequality~\eqref{alpha>=1-inequal} is equivalent to~\eqref{epsilon-varpsilon-ineq} for $\alpha\ge0$ and $\rho\le2$.
\end{remark}

\begin{theorem}\label{alpha>=1-3-thm}
Let $\alpha\ge1$, $x>0$, and $\lambda_{ijk}>0$ for $1\le i\le\ell$, $1\le j\le m$, and $1\le k\le n$. Then
\begin{multline}\label{2:1Sum-Ineq}
\sum_{k=1}^n\sum_{j=1}^m\frac{\bigl(\sum_{i=1}^\ell\lambda_{ijk}\bigr)^{\alpha-1}}{e^{x/\sum_{i=1}^\ell\lambda_{ijk}}-1}
+\sum_{i=1}^\ell\sum_{k=1}^n\frac{\bigl(\sum_{j=1}^m\lambda_{ijk}\bigr)^{\alpha-1}}{e^{x/\sum_{j=1}^m\lambda_{ijk}}-1}\\
+\sum_{j=1}^m\sum_{i=1}^\ell\frac{\bigl(\sum_{k=1}^n\lambda_{ijk}\bigr)^{\alpha-1}}{e^{x/\sum_{k=1}^n\lambda_{ijk}}-1}
\ge3\sum_{k=1}^n\sum_{j=1}^m\sum_{i=1}^\ell\frac{\lambda_{ijk}^{\alpha-1}}{e^{x/\lambda_{ijk}}-1}
\end{multline}
and
\begin{multline}\label{1:2Sum-Ineq}
\sum_{k=1}^n\frac{\bigl(\sum_{j=1}^m\sum_{i=1}^\ell\lambda_{ijk}\bigr)^{\alpha-1}}{e^{x/\sum_{j=1}^m\sum_{i=1}^\ell\lambda_{ijk}}-1}
+\sum_{i=1}^\ell\frac{\bigl(\sum_{k=1}^n\sum_{j=1}^m\lambda_{ijk}\bigr)^{\alpha-1}}{e^{x/\sum_{k=1}^n\sum_{j=1}^m\lambda_{ijk}}-1}\\
+\sum_{j=1}^{m} \frac{\bigl(\sum_{i=1}^\ell\sum_{k=1}^n\lambda_{ijk}\bigr)^{\alpha-1}}{e^{x/\sum_{i=1}^\ell\sum_{k=1}^n\lambda_{ijk}}-1}
\ge3\sum_{k=1}^n\sum_{j=1}^m\sum_{i=1}^\ell\frac{\lambda_{ijk}^{\alpha-1}}{e^{x/\lambda_{ijk}}-1}.
\end{multline}
\end{theorem}

\begin{proof}
As did in the proof of Theorem~\ref{alpha>=1-thm}, we can obtain
\begin{align*}
\mathfrak{H}_\alpha\Biggl(\frac{\sum_{i=1}^\ell\lambda_{ijk}}{x}\Biggr)
&\ge\sum_{i=1}^\ell\mathfrak{H}_\alpha\biggl(\frac{\lambda_{ijk}}{x}\biggr), \\
\mathfrak{H}_\alpha\Biggl(\frac{\sum_{j=1}^m\lambda_{ijk}}{x}\Biggr)
&\ge\sum_{j=1}^m\mathfrak{H}_\alpha\biggl(\frac{\lambda_{ijk}}{x}\biggr),\\
\mathfrak{H}_\alpha\Biggl(\frac{\sum_{k=1}^n\lambda_{ijk}}{x}\Biggr)
&\ge\sum_{k=1}^n\mathfrak{H}_\alpha\biggl(\frac{\lambda_{ijk}}{x}\biggr),\\
\mathfrak{H}_\alpha\Biggl(\frac{\sum_{j=1}^m\sum_{i=1}^\ell\lambda_{ijk}}{x}\Biggr)
&\ge\sum_{j=1}^m\sum_{i=1}^\ell\mathfrak{H}_\alpha\biggl(\frac{\lambda_{ijk}}{x}\biggr),\\
\mathfrak{H}_\alpha\Biggl(\frac{\sum_{k=1}^n\sum_{j=1}^m\lambda_{ijk}}{x}\Biggr)
&\ge\sum_{k=1}^n\sum_{j=1}^m\mathfrak{H}_\alpha\biggl(\frac{\lambda_{ijk}}{x}\biggr),\\
\mathfrak{H}_\alpha\Biggl(\frac{\sum_{i=1}^\ell\sum_{k=1}^n\lambda_{ijk}}{x}\Biggr)
&\ge\sum_{i=1}^\ell\sum_{k=1}^n\mathfrak{H}_\alpha\biggl(\frac{\lambda_{ijk}}{x}\biggr),\\
\sum_{k=1}^n\sum_{j=1}^m\mathfrak{H}_\alpha\Biggl(\frac{\sum_{i=1}^\ell\lambda_{ijk}}{x}\Biggr)
&\ge\sum_{k=1}^n\sum_{j=1}^m\sum_{i=1}^\ell\mathfrak{H}_\alpha\biggl(\frac{\lambda_{ijk}}{x}\biggr), \\
\sum_{i=1}^\ell\sum_{k=1}^n\mathfrak{H}_\alpha\Biggl(\frac{\sum_{j=1}^m\lambda_{ijk}}{x}\Biggr)
&\ge\sum_{i=1}^\ell\sum_{k=1}^n\sum_{j=1}^m\mathfrak{H}_\alpha\biggl(\frac{\lambda_{ijk}}{x}\biggr),\\
\sum_{j=1}^m\sum_{i=1}^\ell\mathfrak{H}_\alpha\Biggl(\frac{\sum_{k=1}^n\lambda_{ijk}}{x}\Biggr)
&\ge\sum_{j=1}^m\sum_{i=1}^\ell\sum_{k=1}^n\mathfrak{H}_\alpha\biggl(\frac{\lambda_{ijk}}{x}\biggr),\\
\sum_{k=1}^n\mathfrak{H}_\alpha\Biggl(\frac{\sum_{j=1}^m\sum_{i=1}^\ell\lambda_{ijk}}{x}\Biggr)
&\ge\sum_{k=1}^n\sum_{j=1}^m\sum_{i=1}^\ell\mathfrak{H}_\alpha\biggl(\frac{\lambda_{ijk}}{x}\biggr),\\
\sum_{i=1}^\ell\mathfrak{H}_\alpha\Biggl(\frac{\sum_{k=1}^n\sum_{j=1}^m\lambda_{ijk}}{x}\Biggr)
&\ge\sum_{i=1}^\ell\sum_{k=1}^n\sum_{j=1}^m\mathfrak{H}_\alpha\biggl(\frac{\lambda_{ijk}}{x}\biggr),\\
\sum_{j=1}^m\mathfrak{H}_\alpha\Biggl(\frac{\sum_{i=1}^\ell\sum_{k=1}^n\lambda_{ijk}}{x}\Biggr)
&\ge\sum_{j=1}^m\sum_{i=1}^\ell\sum_{k=1}^n\mathfrak{H}_\alpha\biggl(\frac{\lambda_{ijk}}{x}\biggr).
\end{align*}
Consequently, it follows that
\begin{multline*}
\sum_{k=1}^n\sum_{j=1}^m\mathfrak{H}_\alpha\Biggl(\frac{\sum_{i=1}^\ell\lambda_{ijk}}{x}\Biggr)
+\sum_{i=1}^\ell\sum_{k=1}^n\mathfrak{H}_\alpha\Biggl(\frac{\sum_{j=1}^m\lambda_{ijk}}{x}\Biggr)\\
+\sum_{j=1}^m\sum_{i=1}^\ell\mathfrak{H}_\alpha\Biggl(\frac{\sum_{k=1}^n\lambda_{ijk}}{x}\Biggr)
\ge\sum_{k=1}^n\sum_{j=1}^m\sum_{i=1}^\ell\mathfrak{H}_\alpha\biggl(\frac{\lambda_{ijk}}{x}\biggr)\\
+\sum_{i=1}^\ell\sum_{k=1}^n\sum_{j=1}^m\mathfrak{H}_\alpha\biggl(\frac{\lambda_{ijk}}{x}\biggr)
+\sum_{j=1}^m\sum_{i=1}^\ell\sum_{k=1}^n\mathfrak{H}_\alpha\biggl(\frac{\lambda_{ijk}}{x}\biggr)
\end{multline*}
and
\begin{multline*}
\sum_{k=1}^n\mathfrak{H}_\alpha\Biggl(\frac{\sum_{j=1}^m\sum_{i=1}^\ell\lambda_{ijk}}{x}\Biggr)
+\sum_{i=1}^\ell\mathfrak{H}_\alpha\Biggl(\frac{\sum_{k=1}^n\sum_{j=1}^m\lambda_{ijk}}{x}\Biggr)\\
+\sum_{j=1}^m\mathfrak{H}_\alpha\Biggl(\frac{\sum_{i=1}^\ell\sum_{k=1}^n\lambda_{ijk}}{x}\Biggr)
\ge\sum_{k=1}^n\sum_{j=1}^m\sum_{i=1}^\ell\mathfrak{H}_\alpha\biggl(\frac{\lambda_{ijk}}{x}\biggr)\\
+\sum_{i=1}^\ell\sum_{k=1}^n\sum_{j=1}^m\mathfrak{H}_\alpha\biggl(\frac{\lambda_{ijk}}{x}\biggr)
+\sum_{j=1}^m\sum_{i=1}^\ell\sum_{k=1}^n\mathfrak{H}_\alpha\biggl(\frac{\lambda_{ijk}}{x}\biggr).
\end{multline*}
Rearranging and simplifying the above two inequalities lead to~\eqref{2:1Sum-Ineq} and~\eqref{1:2Sum-Ineq}.
The proof of Theorem~\ref{alpha>=1-3-thm} is complete.
\end{proof}

\section{A new ratio of many gamma functions and its properties}

\subsection{Preliminaries}
It is common knowledge (\cite[Chapter~6]{abram} and \cite[Chapter~5]{NIST-HB-2010}) that the classical gamma function $\Gamma(z)$ can be defined (\cite{singularity-combined.tex} and~\cite[Chapter~3]{Temme-96-book}) by
\begin{equation*}
\Gamma(z)= \int_{0}^{\infty}t^{z-1}e^{-t} \td t, \quad \Re(z)>0
\end{equation*}
or by
\begin{equation*}
\Gamma(z)=\lim_{n\to\infty}\frac{n!n^z}{\prod_{k=0}^n(z+k)}, \quad z\in\mathbb{C}\setminus\{0,-1,-2,\dotsc\}.
\end{equation*}
Its logarithmic derivative $\psi(z)=[\ln\Gamma(x)]'=\frac{\Gamma'(z)}{\Gamma(z)}$ and $\psi^{(k)}(z)$ for $k\in\mathbb{N}$ are called in sequence digamma function, trigamma function, tetragamma function, and, totally, polygamma functions.
\par
The $q$-gamma function $\Gamma_q(x)$ for $q>0$ and $x>0$, the $q$-analogue of the gamma function $\Gamma(x)$, can be defined (\cite[pp.~493\nobreakdash--496]{aar} and~\cite[Section~1.10]{Basic-hypergeometric-series-2nd}) by
\begin{equation*}
\Gamma_q(x)=
\begin{cases}
(1-q)^{1-x}\prod\limits_{i=0}^\infty\dfrac{1-q^{i+1}}{1-q^{i+x}},& 0<q<1;\\
(q-1)^{1-x}q^{\binom{x}2}\prod\limits_{i=0}^\infty\dfrac{1-q^{-(i+1)}}{1-q^{-(i+x)}},& q>1;\\
\Gamma(x),& q=1.
\end{cases}
\end{equation*}
\par
A real-valued function $F(x)$ defined on a finite or infinite interval $I\subseteq\mathbb{R}$ is said to be completely monotonic on $I$ if and only if $(-1)^kF^{(k)}(x)\ge0$ for all $k\in\{0\}\cup\mathbb{N}$ and $x\in I$. See~\cite[Chapter~XIII]{mpf-1993}, \cite[Chapter~1]{Schilling-Song-Vondracek-2nd}, and~\cite[Chapter~IV]{widder}. A positive function $F(x)$ defined on a finite or infinite interval $I\subseteq\mathbb{R}$ is said to be logarithmically completely monotonic on $I$ if and only if $(-1)^k[\ln F(x)]^{(k)}\ge0$ for all $k\in\mathbb{N}$ and $x\in I$. See~\cite{Atanassov, CBerg, absolute-mon-simp.tex, compmon2, minus-one, e-gam-rat-comp-mon, JAAC384.tex} and~\cite[pp.~66--68, Comments~5.29]{Schilling-Song-Vondracek-2nd}. A nonnegative function $F(x)$ defined on a finite or infinity interval $I$ is called a Bernstein function if its derivative $f'(x)$ is completely monotonic on $I$. See the monograph~\cite{Schilling-Song-Vondracek-2nd}. Among these three concepts, there are the following relations:
\begin{enumerate}
\item
A logarithmically completely monotonic function is completely monotonic, but not conversely. See~\cite{CBerg, absolute-mon-simp.tex, compmon2, minus-one} and~\cite[Theorem~5.11]{Schilling-Song-Vondracek-2nd}.
\item
A completely monotonic function on $(0,\infty)$ or $[0,\infty)$ is equivalent to a Laplace transform. See~\cite[pp.~160--162, Theorems~12a, 12b, and~12c]{widder}.
\item
The reciprocal of a Bernstein function must be logarithmically completely monotonic, but not conversely. See~\cite[pp.~161\nobreakdash--162, Theorem~3]{Chen-Qi-Srivastava-09.tex} and~\cite[p.~64, Proposition~5.25]{Schilling-Song-Vondracek-2nd}.
\end{enumerate}

\subsection{History and backgrounds}
Let $p\in(0,1)$ and $k,n$ be nonnegative integers such that $0\le k\le n$. In~\cite[Theorem]{Alzer-JMAA-2018}, motivated by inequalities related to binomial probability studied in~\cite{Manitoba-Rep-2006, Leblanc-Johnson-JIPAM-2007}, with the help of inequality~\eqref{Alzer-jmaa-2018Lem1}, Alzer proved~\cite{Alzer-JMAA-2018} that the function
\begin{equation}\label{Alzer-G(a)-Eq}
G(x)=\frac{\Gamma(nx+1)}{\Gamma(kx+1)\Gamma((n-k)x+1)} p^{kx}(1-p)^{(n-k)x}
\end{equation}
is completely monotonic on $(0,\infty)$. Indeed, Alzer implicitly proved logarithmically complete monotonicity of $G(x)$ on $(0,\infty)$.
\par
In~\cite[Theorem~2.1]{Ouimet-JMAA-2018} and~\cite{Alzer-CM-JMAA.tex}, with the aid of inequalities~\eqref{Ouimet-Ineq-exp-2018} and~\eqref{Qi-Exp-Ineq-sum-O}, the function $G(x)$ defined in~\eqref{Alzer-G(a)-Eq} and its logarithmically complete monotonicity were generalized as follows.
Let $m\in\mathbb{N}$, $\lambda_i>0$ for $1\le i\le m$, $p_i\in(0,1)$ for $1\le i\le m$, and $\sum_{i=1}^{m}p_i=1$. Then the function
\begin{equation}\label{Qi-calQ(a)-Eq}
\mathcal{Q}(x)=\frac{\Gamma\bigl(1+x\sum_{i=1}^m\lambda_i\bigr)} {\prod_{i=1}^m\Gamma(1+x\lambda_i)}\prod_{i=1}^{m}p_i^{x\lambda_i}
\end{equation}
is logarithmically completely monotonic on $(0,\infty)$. By the way, the conditions in~\cite[Theorem~2.1]{Ouimet-JMAA-2018} are stronger and the conclusion in~\cite[Theorem~2.1]{Ouimet-JMAA-2018} is weaker than corresponding ones in~\cite[Theorem~2.2]{Alzer-CM-JMAA.tex}.
\par
With the help of inequality~\eqref{Qi-Exp-Ineq-sum-3}, the function
\begin{equation*}
\mathcal{Q}_q(x)=\frac{\Gamma_q\bigl(1+x\sum_{i=1}^m\lambda_i\bigr)} {\prod_{i=1}^m\Gamma_q(1+x\lambda_i)}\prod_{i=1}^{m}p_i^{x\lambda_i}
\end{equation*}
for $q\in(0,1)$ and $m\in\mathbb{N}$,
which is the $q$-analogue of the function $\mathcal{Q}(x)$ in~\eqref{Qi-calQ(a)-Eq}, was proved in~\cite{Q-Alzer-CM-Q.tex} to be logarithmically completely monotonic on $(0,\infty)$, where $\lambda_i>0$ for $1\le i\le m$ and $p_i\in(0,1)$ for $1\le i\le m$ with $\sum_{i=1}^{m}p_i=1$.
\par
By virtue of inequality~\eqref{Ouimet-lem-ineq}, Ouimet considered in~\cite[Theorem~2.1]{OUIMET-ARXIV-1907-05262} the function
\begin{equation}\label{g(t)-Ouimet}
g(t)=\frac{\prod_{i=1}^m\Gamma(\nu_it+1)\prod_{j=1}^n\Gamma\bigl(\tau_jt+1\bigr)} {\prod_{i=1}^m\prod_{j=1}^n\Gamma\bigl(\lambda_{ij}t+1\bigr)}
\end{equation}
and its logarithmically complete monotonicity on $(0,\infty)$. But, the conclusion and its proof in~\cite[Theorem~2.1]{OUIMET-ARXIV-1907-05262} are both wrong.
\par
Let $\lambda_{ij}>0$ for $1\le i\le m$ and $1\le j\le n$, let $\nu_i=\sum_{j=1}^n\lambda_{ij}$ and $\tau_j=\sum_{i=1}^m\lambda_{ij}$ for $1\le i\le m$ and $1\le j\le n$, and let
\begin{equation}\label{f(m-n-lambda-alpha-beta-rho)(t)}
f(t) =\frac{\prod_{i=1}^m\Gamma(1+\nu_it)\prod_{j=1}^n\Gamma\bigl(1+\tau_jt\bigr)} {\bigl[\prod_{i=1}^m\prod_{j=1}^n\Gamma\bigl(1+\lambda_{ij}t\bigr)\bigr]^\rho}
\end{equation}
for $\rho\in\mathbb{R}$. This function $f(t)$ is a generalization of the function $g(t)$ in~\eqref{g(t)-Ouimet}. In~\cite[Theorem~4.1]{Ouimet-LCM-BKMS.tex}, with the help of inequality~\eqref{Ouimet-lem-iq}, the following conclusions were obtained:
\begin{enumerate}
\item
when $\rho\le2$, the second derivative $[\ln f(t)]''$ is a completely monotonic function of $t\in(0,\infty)$ and maps from $(0,\infty)$ onto the open interval
$$
\Biggl(0,\frac{\pi^2}{6}\Biggl(\sum_{i=1}^m\nu_i^2 +\sum_{j=1}^n\tau_j^2 -\rho\sum_{i=1}^m\sum_{j=1}^n\lambda_{ij}^2\Biggr)\Biggr);
$$
\item
when $\rho=2$, the logarithmic derivative $[\ln f(t)]' =\frac{f'(t)} {f(t)}$ is a Bernstein function of $t\in(0,\infty)$ and maps from $(0,\infty)$ onto the open interval
\begin{equation*}
\left(0,\ln\frac{\prod_{i=1}^m\nu_i^{\nu_i}\prod_{j=1}^n\tau_j^{\tau_j}} {\Bigl(\prod_{i=1}^m\prod_{j=1}^n\lambda_{ij}^{\lambda_{ij}}\Bigr)^2}\right).
\end{equation*}
\item
when $\rho<2$, the logarithmic derivative $[\ln f(t)]'$ is increasing and concave and maps from $(0,\infty)$ onto the open interval
$$
\Biggl(-\gamma(2-\rho)\sum_{i=1}^m\sum_{j=1}^n\lambda_{ij},\infty\Biggr),
$$
where $\gamma=0.57721566\dotsc$ is the Euler--Mascheroni constant;
\item
when $\rho=2$, the function $f(t)$ is increasing and logarithmically convex and maps from $(0,\infty)$ onto the open interval $(1,\infty)$;
\item
when $\rho<2$, the function $f(t)$ has a unique minimum, is logarithmically convex, and satisfies
\begin{equation*}
\lim_{t\to0^+}f(t)=1 \quad\text{and}\quad \lim_{t\to\infty}f(t)=\infty.
\end{equation*}
\end{enumerate}
\par
Some of the above results have been applied in~\cite{Alzer-JMAA-2018, Ouimet-JMAA-2018, Q-Alzer-CM-Q.tex, Ouimet-LCM-BKMS.tex} to multinomial probability, to the Bernstein estimators on the simplex, to constructing combinatorial inequalities for multinomial coefficients, to constructing inequalities for multivariate beta functions, and the like.

\subsection{Complete monotonicity of a linear combination of finite trigamma functions}

In the paper~\cite{n-psi-sum-CM.tex}, the authors discussed complete monotonicity of the linear combination $\sum_{k=1}^{m}a_k\psi(b_kx+\delta)$ for $\delta\ge0$ and $a_k,b_k>0$.
Now we discuss complete monotonicity of a linear combination of finite trigamma functions.

\begin{theorem}\label{linear-comb-CM-thm}
Let $\lambda_{ij}>0$ for $1\le i\le m$ and $1\le j\le n$, let $\nu_i=\sum_{j=1}^n\lambda_{ij}$ and $\tau_j=\sum_{i=1}^m\lambda_{ij}$ for $1\le i\le m$ and $1\le j\le n$, and let $\rho,\theta\in\mathbb{R}$. If $\rho\le2$ and $\theta\ge0$, then the linear combination
\begin{equation}\label{linear-comb-CM-F}
P(t)=\sum_{i=1}^m\nu_i^{\theta+2}\psi'(1+\nu_it) +\sum_{j=1}^n\tau_j^{\theta+2}\psi'\bigl(1+\tau_jt\bigr) -\rho\sum_{i=1}^m\sum_{j=1}^n\lambda_{ij}^{\theta+2}\psi'\bigl(1+\lambda_{ij}t\bigr)
\end{equation}
is completely monotonic on $(0,\infty)$.
\end{theorem}

\begin{proof}
Employing the integral representation
\begin{equation*}
\psi^{(n)}(z)=(-1)^{n+1}\int_0^\infty\frac{s^n}{1-e^{-s}}e^{-zs}\td s, \quad \Re(z)>0
\end{equation*}
in~\cite[p.~260, 6.4.1]{abram} leads to
\begin{gather*}
P(t)=\sum_{i=1}^m\nu_i^{\theta+2}\int_0^\infty\frac{s}{1-e^{-s}}e^{-(1+\nu_it)s}\td s
+\sum_{j=1}^n\tau_j^{\theta+2}\int_0^\infty\frac{s}{1-e^{-s}}e^{-(1+\tau_jt)s}\td s\\
-\rho\sum_{i=1}^m\sum_{j=1}^n\lambda_{ij}^{\theta+2}\int_0^\infty\frac{s}{1-e^{-s}}e^{-(1+\lambda_{ij}t)s}\td s\\
=\sum_{i=1}^m\nu_i^{\theta+2}\int_0^\infty\frac{s}{e^s-1}e^{-\nu_its}\td s
+\sum_{j=1}^n\tau_j^{\theta+2}\int_0^\infty\frac{s}{e^s-1}e^{-\tau_jts}\td s\\
-\rho\sum_{i=1}^m\sum_{j=1}^n\lambda_{ij}^{\theta+2}\int_0^\infty\frac{s}{e^s-1}e^{-\lambda_{ij}ts}\td s\\
=\sum_{i=1}^m\nu_i^{\theta}\int_0^\infty\frac{u}{e^{u/\nu_i}-1}e^{-tu}\td u
+\sum_{j=1}^n\tau_j^{\theta}\int_0^\infty\frac{u}{e^{u/\tau_j}-1}e^{-tu}\td u\\
-\rho\sum_{i=1}^m\sum_{j=1}^n\lambda_{ij}^{\theta}\int_0^\infty\frac{u}{e^{u/\lambda_{ij}}-1}e^{-tu}\td u\\
=\int_0^\infty\Biggl(\sum_{i=1}^m\frac{\nu_i^{\theta}}{e^{u/\nu_i}-1}
+\sum_{j=1}^n\frac{\tau_j^{\theta}}{e^{u/\tau_j}-1} -\rho\sum_{i=1}^m\sum_{j=1}^{n} \frac{\lambda_{ij}^{\theta}}{e^{u/\lambda_{ij}}-1}\Biggr)ue^{-tu}\td u.
\end{gather*}
Using inequality~\eqref{alpha>=1-inequal} in Theorem~\ref{alpha>=1-thm} shows that, if $\theta\ge0$ and $\rho\le2$, the function $P(t)$ is a Laplace transform of a positive function
$$
\Biggl(\sum_{i=1}^m\frac{\nu_i^{\theta}}{e^{u/\nu_i}-1}
+\sum_{j=1}^n\frac{\tau_j^{\theta}}{e^{u/\tau_j}-1} -\rho\sum_{i=1}^m\sum_{j=1}^{n} \frac{\lambda_{ij}^{\theta}}{e^{u/\lambda_{ij}}-1}\Biggr)u.
$$
Consequently, if $\theta\ge0$ and $\rho\le2$, the function $P(t)$ is completely monotonic on $(0,\infty)$.
The proof of Theorem~\ref{linear-comb-CM-thm} is complete.
\end{proof}

\subsection{A new ratio of many gamma functions and its properties}
Let $\lambda_{ij}>0$ for $1\le i\le m$ and $1\le j\le n$, let $\nu_i=\sum_{j=1}^n\lambda_{ij}$ and $\tau_j=\sum_{i=1}^m\lambda_{ij}$ for $1\le i\le m$ and $1\le j\le n$, and let
\begin{equation}\label{f(m-n-lambda-alpha-beta-rho-theta)(t)}
F(t)=\frac{\prod_{i=1}^{m}[\Gamma(1+\nu_it)]^{\nu_i^\theta}
\prod_{j=1}^n\bigl[\Gamma\bigl(1+\tau_jt\bigr)\bigr]^{\tau_j^\theta}} {\prod_{i=1}^m\prod_{j=1}^{n}\bigl[\Gamma\bigl(1+\lambda_{ij}t\bigr)\bigr]^{\rho\lambda_{ij}^\theta}}
\end{equation}
for $\rho,\theta\in\mathbb{R}$.
It is clear that, when $\theta=0$, the function $F(t)$ becomes $f(t)$ defined in~\eqref{f(m-n-lambda-alpha-beta-rho)(t)}.

\begin{theorem}\label{theta-rho-ratio-thm}
The function $F(t)$ has the following properties:
\begin{enumerate}
\item
If $\rho\le2$ and $\theta\ge0$, the second derivative $[\ln F(t)]''$ is completely monotonic and maps from $(0,\infty)$ onto the interval
\begin{equation*}
\Biggl(0,\frac{\pi^2}{6}\Biggl(\sum_{i=1}^m\nu_i^{\theta+2}+\sum_{j=1}^n\tau_j^{\theta+2} -\rho\sum_{i=1}^m\sum_{j=1}^n\lambda_{ij}^{\theta+2}\Biggr)\Biggr).
\end{equation*}
\item
If $\rho\le2$ and $\theta\ge0$, the logarithmic derivative $[\ln F(t)]'=\frac{F'(t)}{F(t)}$ is increasing and concave on $(0,\infty)$.
\begin{enumerate}
\item
If $\rho=2$ and $\theta=0$, the logarithmic derivative $[\ln F(t)]'=\frac{F'(t)}{F(t)}$ maps $(0,\infty)$ onto
\begin{equation*}
\left(0,\ln\frac{\prod_{i=1}^m\nu_i^{\nu_i}\prod_{j=1}^n\tau_j^{\tau_j}} {\Bigl(\prod_{i=1}^m\prod_{j=1}^n\lambda_{ij}^{\lambda_{ij}}\Bigr)^2}\right)
\end{equation*}
and, consequently, is a Bernstein function on $(0,\infty)$.
\item
If $\rho<2$ or $\theta>0$, the logarithmic derivative $[\ln F(t)]'=\frac{F'(t)}{F(t)}$ maps $(0,\infty)$ onto
\begin{equation*}
\Biggl(-\gamma\Biggl(\sum_{i=1}^m\nu_i^{\theta+1}+\sum_{j=1}^n\tau_j^{\theta+1} -\rho\sum_{i=1}^m\sum_{j=1}^n\lambda_{ij}^{\theta+1}\Biggr), \infty\Biggr),
\end{equation*}
where $\gamma=0.57721566\dotsc$ is the Euler--Mascheroni constant.
\end{enumerate}
\item
If $\rho\le2$ and $\theta\ge0$, the function $F(t)$ is logarithmically convex on $(0,\infty)$.
\begin{enumerate}
\item
If $\rho=2$ and $\theta=0$, the function $F(t)$ is increasing and maps $(0,\infty)$ onto $(1,\infty)$.
\item
If $\rho<2$ or $\theta>0$, the function $F(t)$ has a unique minimum and the limits
\begin{equation*}
\lim_{t\to0^+}F(t)=1 \quad\text{and}\quad \lim_{t\to\infty}F(t)=\infty.
\end{equation*}
\end{enumerate}
\end{enumerate}
\end{theorem}

\begin{proof}
Taking the logarithm of $F(t)$ in~\eqref{f(m-n-lambda-alpha-beta-rho-theta)(t)} and differentiating give
\begin{gather*}
\ln F(t)=\sum_{i=1}^m\nu_i^\theta\ln\Gamma(1+\nu_it)+\sum_{j=1}^n\tau_j^\theta\ln\Gamma\bigl(1+\tau_jt\bigr) -\rho\sum_{i=1}^m\sum_{j=1}^n\lambda_{ij}^\theta\ln\Gamma\bigl(1+\lambda_{ij}t\bigr),\\
[\ln F(t)]'=\sum_{i=1}^m\nu_i^{\theta+1}\psi(1+\nu_it)+\sum_{j=1}^n\tau_j^{\theta+1}\psi\bigl(1+\tau_jt\bigr) -\rho\sum_{i=1}^m\sum_{j=1}^n\lambda_{ij}^{\theta+1}\psi\bigl(1+\lambda_{ij}t\bigr),
\end{gather*}
and $[\ln F(t)]''=P(t)$, where $P(t)$ is defined in~\eqref{linear-comb-CM-F}. From Theorem~\ref{linear-comb-CM-thm}, it follows immediately that, if $\rho\le2$ and $\theta\ge0$, the second derivative $[\ln F(t)]''$ is completely monotonic on $(0,\infty)$ and, consequently, that the function $F(t)$ is logarithmically convex on $(0,\infty)$.
\par
It is easy to obtain that
\begin{equation*}
\lim_{t\to0^+}F(t)=1, \quad \lim_{t\to0^+}[\ln F(t)]'=-\gamma\Biggl(\sum_{i=1}^m\nu_i^{\theta+1}+\sum_{j=1}^n\tau_j^{\theta+1} -\rho\sum_{i=1}^m\sum_{j=1}^n\lambda_{ij}^{\theta+1}\Biggr),
\end{equation*}
and
\begin{equation*}
\lim_{t\to0^+}[\ln F(t)]''=\lim_{t\to0^+}P(t)=\frac{\pi^2}{6}\Biggl(\sum_{i=1}^m\nu_i^{\theta+2}+\sum_{j=1}^n\tau_j^{\theta+2} -\rho\sum_{i=1}^m\sum_{j=1}^n\lambda_{ij}^{\theta+2}\Biggr).
\end{equation*}
\par
Since $P(t)=[\ln F(t)]''$ is completely monotonic on $(0,\infty)$, the logarithmic derivative $[\ln F(t)]'$ is increasing and concave on $(0,\infty)$. Utilizing $\lim_{t\to\infty}[\psi(t)-\ln t]=0$ in~\cite[Theorem~1]{theta-new-proof.tex-BKMS} and~\cite[Section~1.4]{Sharp-Ineq-Polygamma-Slovaca.tex} produces
\begin{align*}
\lim_{t\to\infty}[\ln F(t)]'&=\sum_{i=1}^m\nu_i^{\theta+1}\lim_{t\to\infty}[\psi(1+\nu_it)-\ln(1+\nu_it)]\\
&\quad+\sum_{j=1}^n\tau_j^{\theta+1}\lim_{t\to\infty}\bigl[\psi\bigl(1+\tau_jt\bigr)-\ln\bigl(1+\tau_jt\bigr)\bigr]\\
&\quad-\rho\sum_{i=1}^m\sum_{j=1}^n\lambda_{ij}^{\theta+1} \lim_{t\to\infty}\bigl[\psi\bigl(1+\lambda_{ij}t\bigr)-\ln\bigl(1+\lambda_{ij}t\bigr)\bigr]\\
&\quad+\ln\lim_{t\to\infty}\frac{\prod_{i=1}^{m}(1+\nu_it)^{\nu_i^{\theta+1}} \prod_{j=1}^n\bigl(1+\tau_jt\bigr)^{\tau_j^{\theta+1}}} {\prod_{i=1}^m\prod_{j=1}^n\bigl(1+\lambda_{ij}t\bigr)^{\rho\lambda_{ij}^{\theta+1}}}\\
&=\ln\lim_{t\to\infty}\frac{\prod_{i=1}^{m}(1/t+\nu_i)^{\nu_i^{\theta+1}} \prod_{j=1}^n\bigl(1/t+\tau_j\bigr)^{\tau_j^{\theta+1}}} {\prod_{i=1}^m\prod_{j=1}^n\bigl(1/t+\lambda_{ij}\bigr)^{\rho\lambda_{ij}^{\theta+1}}}\\
&\quad+\ln\lim_{t\to\infty}\frac{\prod_{i=1}^{m}t^{\nu_i^{\theta+1}} \prod_{j=1}^{n}t^{\tau_j^{\theta+1}}} {\prod_{i=1}^m\prod_{j=1}^{n}t^{\rho\lambda_{ij}^{\theta+1}}}\\
&=\ln\frac{\prod_{i=1}^m\nu_i^{\nu_i^{\theta+1}} \prod_{j=1}^n\tau_j^{\tau_j^{\theta+1}}} {\prod_{i=1}^m\prod_{j=1}^n\lambda_{ij}^{\rho\lambda_{ij}^{\theta+1}}}+\ln\lim_{t\to\infty}\frac{\prod_{i=1}^{m}t^{\nu_i^{\theta+1}} \prod_{j=1}^{n}t^{\tau_j^{\theta+1}}} {\prod_{i=1}^m\prod_{j=1}^{n}t^{\rho\lambda_{ij}^{\theta+1}}},
\end{align*}
where the last term is equal to
\begin{equation*}
\ln\lim_{t\to\infty}t^{\sum_{i=1}^m\nu_i^{\theta+1}+\sum_{j=1}^n\tau_j^{\theta+1} -\rho\sum_{i=1}^m\sum_{j=1}^n\lambda_{ij}^{\theta+1}}
=
\begin{dcases}
0, & \theta=0 \text{ and }\rho=2;\\
\infty, & \theta>0 \text{ or } \rho<2.
\end{dcases}
\end{equation*}
\par
By virtue of the formula
\begin{equation*}
\ln\Gamma(z+1)=\biggl(z+\frac12\biggr)\ln z-z+\frac12\ln(2\pi)+\int_{0}^{\infty}\vartheta(s)e^{-zs}\td s
\end{equation*}
in~\cite[p.~62, (3.20)]{Temme-96-book}, where
\begin{equation*}
\vartheta(s)=\frac1s\biggl(\frac1{e^s-1}-\frac1s+\frac12\biggr),
\end{equation*}
we can find
\begin{gather*}
\ln F(t)=\sum_{i=1}^m\nu_i^\theta\biggl[\biggl(\nu_it+\frac12\biggr)\ln(\nu_it) -\nu_it+\frac12\ln(2\pi)+\int_{0}^{\infty}\vartheta(s)e^{-\nu_its}\td s\biggr]\\
+\sum_{j=1}^n\tau_j^\theta\biggl[\biggl(\tau_jt+\frac12\biggr)\ln(\tau_jt)-\tau_jt +\frac12\ln(2\pi)+\int_{0}^{\infty}\vartheta(s)e^{-\tau_jts}\td s\biggr]\\ -\rho\sum_{i=1}^m\sum_{j=1}^n\lambda_{ij}^\theta\biggl[\biggl(\lambda_{ij}t+\frac12\biggr)\ln(\lambda_{ij}t) -\lambda_{ij}t+\frac12\ln(2\pi)+\int_{0}^{\infty}\vartheta(s)e^{-\lambda_{ij}ts}\td s\biggr]\\
=\int_{0}^{\infty}\vartheta(s)\Biggl[\sum_{i=1}^m\nu_i^\theta e^{-\nu_its} +\sum_{j=1}^n\tau_j^\theta e^{-\tau_jts} -\rho\sum_{i=1}^m\sum_{j=1}^n\lambda_{ij}^\theta e^{-\lambda_{ij}ts}\Biggr]\td s\\
+\frac12\ln\frac{\prod_{i=1}^m\nu_i^{\nu_i^\theta} \prod_{j=1}^n\tau_j^{\tau_j^\theta}} {\prod_{i=1}^m\prod_{j=1}^n\lambda_{ij}^{\rho\lambda_{ij}^\theta}}
+\Biggl[\sum_{i=1}^m\nu_i^\theta+\sum_{j=1}^n\tau_j^\theta
-\rho\sum_{i=1}^m\sum_{j=1}^n\lambda_{ij}^\theta\Biggr]\frac{\ln(2\pi t)}2\\
+\left[\ln\frac{\prod_{i=1}^m\nu_i^{\nu_i^{\theta+1}} \prod_{j=1}^n\tau_j^{\tau_j^{\theta+1}}} {\Bigl(\prod_{i=1}^m\prod_{j=1}^n\lambda_{ij}^{\lambda_{ij}^{\theta+1}}\Bigr)^\rho} -\Biggl(\sum_{i=1}^m\nu_i^{\theta+1}+\sum_{j=1}^n\tau_j^{\theta+1} -\rho\sum_{i=1}^m\sum_{j=1}^n\lambda_{ij}^{\theta+1}\Biggr)\right]t\\
+\Biggl[\sum_{i=1}^m\nu_i^{\theta+1}+\sum_{j=1}^n\tau_j^{\theta+1} -\rho\sum_{i=1}^m\sum_{j=1}^n\lambda_{ij}^{\theta+1}\Biggr]t\ln t
\to\infty, \quad t\to\infty
\end{gather*}
where, when $\rho=2$ and $\theta=0$, we used the fact~\cite{Ouimet-LCM-BKMS.tex} that
\begin{gather*}
\frac{\prod_{i=1}^m\nu_i^{\nu_i}\prod_{j=1}^n\tau_j^{\tau_j}} {\Bigl(\prod_{i=1}^m\prod_{j=1}^n\lambda_{ij}^{\lambda_{ij}}\Bigr)^2}
=\frac{\prod_{i=1}^m\nu_i^{\nu_i}}{\prod_{i=1}^m\prod_{j=1}^n\lambda_{ij}^{\lambda_{ij}}} \frac{\prod_{j=1}^n\tau_j^{\tau_j}}{\prod_{i=1}^m\prod_{j=1}^n\lambda_{ij}^{\lambda_{ij}}}\\
=\prod_{i=1}^m\frac{\nu_i^{\nu_i}}{\prod_{j=1}^n\lambda_{ij}^{\lambda_{ij}}} \prod_{j=1}^n\frac{\tau_j^{\tau_j}}{\prod_{i=1}^m\lambda_{ij}^{\lambda_{ij}}}
=\prod_{i=1}^m\frac{\prod_{j=1}^{n}\bigl(\sum_{\ell=1}^{n}\lambda_{i\ell}\bigr)^{\lambda_{ij}}} {\prod_{j=1}^n\lambda_{ij}^{\lambda_{ij}}}
\prod_{j=1}^n\frac{\prod_{i=1}^m\bigl(\sum_{\ell}^m\lambda_{\ell j}\bigr)^{\lambda_{j\ell}}} {\prod_{i=1}^m\lambda_{ij}^{\lambda_{ij}}}\\
=\prod_{i=1}^m\prod_{j=1}^n\Biggl(\frac{\sum_{\ell=1}^n\lambda_{i\ell}}{\lambda_{ij}}\Biggr)^{\lambda_{ij}} \prod_{j=1}^n\prod_{i=1}^m\Biggl(\frac{\sum_{\ell}^m\lambda_{\ell j}}{\lambda_{ij}}\Biggr)^{\lambda_{j\ell}}
>1\times1=1.
\end{gather*}
The proof of Theorem~\ref{theta-rho-ratio-thm} is complete.
\end{proof}

\section{Four functions to be investigated}

Finally, basing on inequalities~\eqref{2:1Sum-Ineq} and~\eqref{1:2Sum-Ineq} in Theorem~\ref{alpha>=1-3-thm}, motivated by Theorems~\ref{linear-comb-CM-thm} and~\ref{theta-rho-ratio-thm}, we would like to suggest to consider two linear combinations $L_1(t), L_2(t)$ and two ratios $R_1(t), R_2(t)$ defined by
\begin{gather*}
L_1(t)=\sum_{i=1}^\ell\sum_{j=1}^m \Biggl(\sum_{k=1}^n\lambda_{ijk}\Biggr)^\theta\psi'\Biggl(1+t\sum_{k=1}^n\lambda_{ijk}\Biggr)\\
+\sum_{j=1}^m\sum_{k=1}^n \Biggl(\sum_{i=1}^\ell\lambda_{ijk}\Biggr)^\theta \psi'\Biggl(1+t\sum_{i=1}^\ell\lambda_{ijk}\Biggr)\\
+\sum_{k=1}^n\sum_{i=1}^\ell \Biggl(\sum_{j=1}^m\lambda_{ijk}\Biggr)^\theta \psi'\Biggl(1+t\sum_{j=1}^m\lambda_{ijk}\Biggr)
-\rho\sum_{i=1}^\ell\sum_{j=1}^m\sum_{k=1}^{n}\lambda_{ijk}^\theta\psi'\bigl(1+\lambda_{ijk}t\bigr),\\
L_2(t)=\sum_{i=1}^\ell \Biggl(\sum_{j=1}^m\sum_{k=1}^n\lambda_{ijk}\Biggr)^\theta\psi'\Biggl(1+t\sum_{j=1}^m\sum_{k=1}^n\lambda_{ijk}\Biggr)\\
+\sum_{j=1}^m \Biggl(\sum_{k=1}^n\sum_{i=1}^\ell\lambda_{ijk}\Biggr)^\theta\psi'\Biggl(1+t\sum_{k=1}^n\sum_{i=1}^\ell\lambda_{ijk}\Biggr)\\
+\sum_{k=1}^n \Biggl(\sum_{i=1}^\ell\sum_{j=1}^m\lambda_{ijk}\Biggr)^\theta\psi'\Biggl(1+t\sum_{i=1}^\ell\sum_{j=1}^m\lambda_{ijk}\Biggr)
-\rho\sum_{i=1}^\ell\sum_{j=1}^m\sum_{k=1}^{n} \lambda_{ijk}^\theta\psi'\bigl(1+\lambda_{ijk}t\bigr),\\
R_1(t)=\frac{\left(\begin{gathered}\prod_{i=1}^\ell\prod_{j=1}^m \Biggl[\Gamma\Biggl(1+t\sum_{k=1}^n\lambda_{ijk}\Biggr)\Biggr]^{(\sum_{k=1}^n\lambda_{ijk})^\theta}\\
\times\prod_{j=1}^m\prod_{k=1}^n \Biggl[\Gamma\Biggl(1+t\sum_{i=1}^\ell\lambda_{ijk}\Biggr)\Biggr]^{(\sum_{i=1}^\ell\lambda_{ijk})^\theta}\\
\times\prod_{k=1}^n\prod_{i=1}^\ell \Biggl[\Gamma\Biggl(1+t\sum_{j=1}^m\lambda_{ijk}\Biggr)\Biggr] ^{(\sum_{j=1}^m\lambda_{ijk})^\theta}\end{gathered}\right)}
{\displaystyle\prod_{i=1}^\ell\prod_{j=1}^m\prod_{k=1}^{n}\bigl[\Gamma\bigl(1+\lambda_{ijk}t\bigr)\bigr]^{\rho\lambda_{ijk}^\theta}},\\
R_2(t)=\frac{\left(\begin{gathered}\prod_{i=1}^\ell \Biggl[\Gamma\Biggl(1+t\sum_{j=1}^m\sum_{k=1}^n\lambda_{ijk}\Biggr)\Biggr] ^{(\sum_{j=1}^m\sum_{k=1}^n\lambda_{ijk})^\theta}\\
\times\prod_{j=1}^m\Biggl[\Gamma\Biggl(1+t\sum_{k=1}^n\sum_{i=1}^\ell\lambda_{ijk}\Biggr)\Biggr]^{(\sum_{k=1}^n\sum_{i=1}^\ell\lambda_{ijk})^\theta}\\
\times\prod_{k=1}^n \Biggl[\Gamma\Biggl(1+t\sum_{i=1}^\ell\sum_{j=1}^m\lambda_{ijk}\Biggr)\Biggr] ^{(\sum_{i=1}^\ell\sum_{j=1}^m\lambda_{ijk})^\theta}\end{gathered}\right)}
{\displaystyle\prod_{i=1}^\ell\prod_{j=1}^m\prod_{k=1}^{n}\bigl[\Gamma\bigl(1+\lambda_{ijk}t\bigr)\bigr]^{\rho\lambda_{ijk}^\theta}},
\end{gather*}
where $t>0$, $\theta\ge0$, and $\lambda_{ijk}>0$ for $1\le i\le\ell$, $1\le j\le m$, and $1\le k\le n$.

For the sake of saving the space and shortening the length of this paper, we would not like to write down our guesses on possible conclusions and their detailed proofs of the functions $L_1(t)$, $L_2(t)$, $R_1(t)$, and $R_2(t)$ on $(0,\infty)$.

\section{Remarks}

\begin{remark}
The function $\frac{1}{e^x-1}$ has been investigated from viewpoints of analytic combinatorics and analytic number theory in~\cite{Eight-Identy-More.tex, exp-derivative-sum-Combined.tex, Exp-Diff-Ratio-Wei-Guo.tex, CAM-D-13-01430-Xu-Cen} and closely related references therein.
There have been so many papers dedicated to research of functions involving the exponential function, please refer to the papers~\cite{Best-Constant-exponential.tex, best-constant-one-simple.tex, absolute-mon-simp.tex, note-on-li-chen-conj-I.tex, Guo-Qi-MJMS-15.tex, mon-element-exp-final.tex, power-exp-new-proofs-miq, property-psi-ii-Munich.tex, best-constant-one.tex, Miao-Liu-Qi-jipam-08, exp-beograd, deg4-exp.tex, new-inequality-qi.tex, Bessel-ineq-Dgree-CM.tex, Bell-Polyn-P2v.tex, comp-mon-element-exp.tex, QiBerg.tex, qi-guo-taiwanese, power-exp-qi-debnath, Mansour-Qi-gam-tan.tex, Bell-Poly-Gen-P.tex, AJOM-D-16-00138.tex, Int-Mean-Ineq-Exp-Log.tex, 195-2017-JOCAAA.tex, simp-exp-degree-revised.tex, simp-exp-degree-new.tex, best-constant-one-simple-real.tex, exp-reciprocal-cm-IJOPCM.tex} and closely related references therein.
\end{remark}

\begin{remark}
There are so many papers dedicated to study of ratios of many gamma functions, please refer to the papers~\cite{Open-TJM-2003-Ineq-Ext-JAT.tex, Guo-Qi-TJM-03.tex, bounds-two-gammas.tex, Gautschi-Kershaw-TJANT.tex, Qi-Agar-Surv-JIA.tex, ratio-sqrt-gamma-indonesia.tex, JAAC384.tex, Wendel2Elezovic.tex-JIA, Wendel-Gautschi-type-ineq-Banach.tex, Mortici-aar.tex} and closely related references therein.
\end{remark}

\end{document}